\documentclass[10pt,a4paper,fleqn]{article}

\usepackage{amsfonts}
\usepackage[numbers]{natbib}
\usepackage{amsmath}
\usepackage{amsthm}

\usepackage[math]{kurier}
\usepackage[stable]{footmisc}

\newtheorem{thm}{Theorem}[section]
\newtheorem{lemma}[thm]{Lemma}

\newtheorem{rem}[thm]{Remark}
\newtheorem{cor}[thm]{Corollary}

\newcommand{\reals}{\mathbb{R}}
\newcommand{\ex}{\mathbb{E}}
\newcommand{\Ex}[1]{\mathbb{E}\left[\,#1\,\right]}
\newcommand{\Fc}{\mathcal{F}}

\newcommand{\Ac}{\mathcal{A}}
\newcommand{\Zc}{\mathcal{Z}}
\newcommand{\Oc}{\mathcal{O}}

\newcommand{\olo}{\int_{t_i}^{t_{i+1}}}

\newcommand{\dHalf}[1]{\frac{\delta_{#1}}{2}}

\newcommand{\Ypi}[1]{Y^{\pi}_{2,t_{#1}}}
\newcommand{\Zpi}[1]{Z^{\pi}_{2,t_{#1}}}

\newcommand{\yEuler}[1]{Y^{\pi}_{1,t_{#1}}} 
\newcommand{\zEuler}[1]{Z^{\pi}_{1,t_{#1}}}

\newcommand{\dY}[1]{\Delta^{\pi}_{#1} Y}
\newcommand{\dZ}[1]{\Delta^{\pi}_{#1} Z}
\newcommand{\df}[1]{\Delta^{\pi}_{#1} f}

\newcommand{\abs}[1]{\left|\,#1\,\right|}
\newcommand{\paren}[1]{\left(\,#1\,\right)}
\newcommand{\brac}[1]{\left[\,#1\,\right]}
\newcommand{\norm}[1]{\left\|#1\right\|}
\newcommand{\dB}[1]{\Delta W_{i+1}^{#1}}
\newcommand{\dt}[1]{\delta_{#1}}

\numberwithin{equation}{section}

\begin{document}

\title{Second order discretization of Backward SDEs\footnote{The results in this paper were announced by the second author at the workshop "New advances in Backward SDEs for financial engineering applications", Tamerza (Tunisia), October  25 - 28, 2010}} 
\author{D. Crisan, K. Manolarakis\thanks{%
Department of Mathematics, Imperial College London, 180 Queen's Gate,
London, SW7 2AZ, UK.}}
\date{}
\maketitle

\begin{abstract} 
In \cite{CriMan} the authors suggested a new algorithm for the numerical approximation of a BSDE by merging the cubature method with the first order discretization
developed by \citet{BouchardTouzi} and \citet{Z}. Though the algorithm presented in [5] compared satisfactorily with other methods it lacked the higher order nature of the cubature method 
due to the use of the low order discretization. In this paper we  introduce a second order discretization of the BSDE in the spirit of higher order 
implicit-explicit schemes for forward SDEs and predictor corrector methods. 
\vspace{5mm}


\vspace{5mm}


\end{abstract} 

\section{Introduction}
The present paper is concerned with the problem of numerical approximation to forward backward SDEs (FBSDEs henceforth). Let $(\Omega,{\cal F},P)$ be a probability space on which we have defined a triple of processes $(X,Y,Z)$ which solve the decoupled forward-backward system:
\begin{equation}
\begin{split}
&X_t = X_0 +\int_0^t V_0\left(X_u\right) du + \sum_{i=1}^d \int_0^t V_i\left(X_u\right)\circ dW^i_u \\  
&Y_{t}=\Phi(X_T) +\int_{t}^{T}f(X_u ,Y_{u},Z_{u})du-\int_{t}^{T}Z_{u}\cdot dW_{u}\label{BSDE1}
\end{split},\quad t\in \lbrack 0,T]\ .  
\end{equation}
where $W_t$ is a $d$-dimensional Brownian motion and 
\[
V_k: \reals^q \rightarrow \reals^q,\, k=0 ,\ldots,d,\quad f:\reals^q\times \reals \times \reals^d\rightarrow \reals
\] 
are some appropriate functions.  
The system is called decoupled as the (backward) processes $(Y,Z)$ do not appear in the dynamics of the forward component $X$. 
Systems of the form \eqref{BSDE1} have received a lot of attention over the past twenty years primarily due to their applications in the field of  Mathematical Finance (see for example \cite{m2an} and the references therein for details).

Of equal importance is the fact that, the stochastic flow associated with \eqref{BSDE1}, i.e., the triple of processes  $(X^{(t,x)},Y^{(t,x)},Z^{(t,x)})$, $(t,x)\in [0,T]\times  \reals$ satisfying 
\begin{equation}
\begin{split}
&X_s^{t,x} = x +\int_t^s V_0\left(X_u^{t,x}\right) du + \sum_{i=1}^d \int_t^s V_i\left(X_u^{t,x}\right)\circ dW^i_u \\  
&Y_{s}^{t,x}=\Phi(X_T^{t,x}) +\int_{s}^{T}f(X_u ^{t,x},Y_{u}^{t,x},Z_{u}^{t,x})du-\int_{t}^{T}Z_{u}^{t,x}\cdot dW_{u}  \label{flowBSDE1}
\end{split}\quad s\in \lbrack t,T],
\end{equation}
provides a Feynman-Kac representation for the (viscosity) solution of a class of semi-linear
PDEs. In particular, let $u(t,x):[0,T]\times \reals^d \rightarrow \reals^d$, be the viscosity solution of 
\begin{equation}
\label{b_pde}
\begin{split}
&\frac{du}{dt} +\tilde{V}\left( x\right) \cdot \nabla u+\frac{1}{2}Tr\left[
V\left( x\right) V^{\ast }\left( x\right) D^{2}u\right] +f\left(
t,x,u,\nabla uV\left( x\right) \right) =0 \\  
&u(T,x) = \Phi(x)
\end{split},
\end{equation}
where 
%\begin{equation}
$\tilde{V}\left( x\right) =V_{0}(x)-\frac{1}{2}\sum_{j=1}^{d}\nabla
V_{j}(x)V_{j}(x)$.  
%\label{Vtilde}\end{equation}     
In their seminal work \cite{PP2}, Pardoux and Peng showed that 
\begin{equation}
u(t,x) = Y^{t,x}_t,\quad (t,x) \in  [0,T]\times \reals^d.
\end{equation}
In addition, \citet{Mazhang} showed that when this viscosity solution is continuously differentiable in its spatial variables, we have the following  representation  
for the solution of the PDE \ref{b_pde} and its gradient
\begin{equation}
\label{FKrep}
u(t,x) = Y^{t,x}_t,\quad Z^{t,x}_t = \nabla u(t,x) V(x)  ,\quad a.s.
\end{equation}
From the perspective of numerical analysis, the above means that any probabilistic method for the resolution of \eqref{BSDE1}, provides an algorithm 
for the resolution of semi-linear PDEs.
As a result, the interest in robust algorithms  for their resolution is high.   
  
Algorithms designed to solve the problem of numerical approximation of the backward part of \eqref{BSDE1} consist of two parts: 
Firstly, the backward equation is discretized. This step involves the use of one or more   conditional expectations. Secondly, a numerical method is grafted onto the chosen discretization to compute the conditional expectations involved. Following the work of \citet{BouchardTouzi} and \citet{Z}, the most popular discretization scheme is given by:\footnote{See \citet{benderzhang} for an alternative approach based on Picard iterations} 
\begin{equation}
\label{BkwdEuler}
\begin{split}
&\yEuler{n} := \Phi(X_{t_n}),\quad \zEuler{n}:=0 \\
&\zEuler{i}:= \frac{1}{\dt{i+1}}
\ex\brac{\yEuler{i+1} \Delta W_{i+1}|\Fc_{t_i}},\quad i=0,\ldots,n-1\\
&\yEuler{i}:= \ex\brac{\yEuler{i+1}|\Fc_{t_i}} + f\paren{t_i,X_{t_i},\yEuler{i},\zEuler{i}},\quad i=0,\ldots,n-1,
\end{split}
\end{equation}
where $\pi$ is a given partition $\pi:=\{0=t_0<t_1<\ldots<t_n=T\}$ of $[0,T]$ and $\Delta W_{i+1}=W_{t_{i+1}}-W_{t_i}$, $i=0,\ldots,n-1$.

Assuming that all coefficients of \eqref{BSDE1} are at least Lipschitz continuous in their spatial variables,  we have, following \cite{BouchardTouzi} and \cite{Z}, that
\begin{equation}
\label{BkwdEulerEst}
\sup_{0\le t\le T} \ex\brac{\abs{Y^{\pi}_t -Y_t}^2} +\ex\brac{\int_0^T \abs{Z^{\pi}_t -Z_t}^2 dt } \le C |\pi|,
\end{equation}
where $\{(Y^\pi_t,Z^\pi_t),t\ge 0\}$ are the step processes
\begin{equation}
\label{stepEuler} 
Y^{\pi}_t:= \sum_{i=0}^{n-1}\yEuler{i}\, 1_{[t_i,t_{i+1})}(t) +  \yEuler{n}\, 1_{t=t_n},\quad Z^{\pi}_t:= \sum_{i=0}^{n-1}\zEuler{i} \,1_{[t_i,t_{i+1})}(t) +  \zEuler{n}\, 1_{t=t_n}.
\end{equation}
In other words, the above discretization of the backward part achieves a convergence of order $1/n^{1/2}$, when $n$ points are used, i.e. the same order 
 as the strong convergence order of the Euler scheme for a (classical) SDE.  When more smoothness on the coefficients is assumed, it is shown in  
\citet{gobetlabart} that the rate of convergence 
 of the processes  \eqref{stepEuler} is of order $1/n$. In fact an error expansion is obtained in \cite{gobetlabart} and the leading order coefficients in the error expansion are identified. 

To be more precise, the above results are proved for the case when the process $X$ in (\ref{BkwdEuler}) is replaced by its Euler approximation 
$X^\pi$ and, respectively, the filtration $\{\Fc_{t_i}\}_{i=0}^n$     
is replaced by the natural filtration $\{\Fc^{\pi}_{t_i}\}_{i=0}^n$ 
associated   with $X^\pi$. However the same proofs apply both to $X^\pi$ and $X$.  

To obtain an implementable scheme using the discretization \eqref{BkwdEuler}, one has to develop a method for approximating the conditional expectations\footnote{When the process $X$ in (\ref{BkwdEuler}) is replaced by its Euler approximation, then 
$\ex[\yEuler{i+1}|\Fc^\pi_{t_i}] $ and $\ex[\yEuler{i+1}  \Delta W_{i+1}|\Fc^\pi_{t_i}]$ need to be computed.}  $\ex[\yEuler{i+1}|\Fc_{t_i}] $ and $\ex[\yEuler{i+1}  \Delta W_{i+1}|\Fc_{t_i}]$. Various such methods have been introduced, based on Malliavin calculus \cite{BouchardTouzi}, on projection on function basis \cite{gobetlemorwarin}, \cite{gobetlemorwarin2} and on 
quantization \cite{ballypages}. In \cite{CriMan}, the authors suggested the application of the cubature method of \cite{LV}, which is  based on the ideas of \citet{ Kusuoka}.
The overall rate of convergence of this second approximation step is still of order $1/n^{1/2}$ or of order $1/n$ when coefficients are smooth. 

In this paper, we introduce a discretization for the backward component  of the  BSDEs of order 
$1/n^2$. As for the discretization \eqref{BkwdEuler}, this new scheme will require the computation of conditional expectations.
This can be done by means of an order 5  cubature formula. Since the order 5 cubature method has local order error $1/n^3$,  when combined with a second order discretization for BSDEs, will generate a genuine 
second order algorithm for BSDEs.

To understand which terms such a scheme should include we revisit 
\eqref{BkwdEuler} from the point of view of Stratonovich-Taylor expansions. Moreover, to deduce the rate of convergence of the method, we shall rely on the Feynman Kac representation \eqref{FKrep}
and on some estimates on the bounds for the derivatives of \eqref{b_pde} obtained by \citet{delarue}. 

\section{Preliminaries}

Throughout the paper, we will use  the following assumptions:
\begin{description}
\item[(A)]  The coefficients of the forward SDE $V_{i}:\mathbb{R}%
^{d}\rightarrow \mathbb{R}^{d}$, $i=0,1,...,d\;$have all entries belonging
to $C_{b}^{\infty }(\mathbb{R}^{d})$, the space of bounded infinitely
differentiable functions with all partial derivatives bounded. We also
assume that the matrix $V:=(V_{1}|\ldots |V_{d})$ is elliptic.

\item[(B)]  The driver of the BSDE $f:[0,1]\times \mathbb{R}^{d}\times 
\mathbb{R}\times \mathbb{R}^{d}\rightarrow \mathbb{R}$ belongs to $%
C_{b}^{[m/2],m}$. The exact value for the parameter $m$ shall be made precise as we proceed.
\item[(C1)] The terminal condition $\Phi:\reals^d\rightarrow \reals$ is Lipschitz continuous.  
\item[(C2)] The terminal condition $\Phi$ belongs to $ C_b^m\left(\reals^d;\reals\right)$. Again the value of $m$ shall be determined further on. 

\end{description}
We shall denote by $K$ the constant that bounds all derivatives that appear in our assumptions. When (C1) is in force
we shall assume that $\Phi$ is  $K$-Lipschitz. 
Of course under (A), (B) with $m\ge 1$ and (C1) the system \eqref{BSDE1} has a unique solution such that
\[
\Ex{\sup_{0\le t\le T}\left(Y_t^2 +|X_t|^2\right) +\int_0^T |Z_s|^2 ds} <\infty
\]

To abbreviate notation we shall denote by $M_i \equiv M_{t_i},\,i=0,\ldots,n$,  where $M$ can be any of the processes that appear in this paper. Moreover we write $\ex_s\brac{\cdot}$ for $ \ex\brac{\cdot|\Fc_s}$. Note that in the current (Markovian) set up $\ex\brac{\cdot|\Fc_s} = \ex\brac{\cdot|X_s}.$
We shall also consider as given a partition $\pi:=\{0=t_0<t_1<\ldots<t_n=T\}$ of $[0,T]$ and consider the related notation
\[
\dt{i} := t_{i} -t_{i-1},\quad \Delta W_{i}:= W_i - W_{i-1},\quad \Delta W^l_{i}:= W^l_i - W^l_{i-1},\quad l=1,\ldots,d,\quad i=1,\ldots,n.
\] 

Lastly, the driver of the BSDE shall be abbreviated as $\bar{f}(s,x) = f(x,u(s,x),\nabla u V (s,x))$, where $u$ is the solution of the semilinear PDE \eqref{b_pde}.

\begin{rem}
Any attempt to discretize the backward part of \eqref{BSDE1} should start with a method to produce (an approximation for ) the value of the forward diffusion 
at the times on the partition $\pi$. In this paper we assume that we have at our disposal the forward diffusion $X$ at the times of the partition $X_{t_1},\ldots,X_{t_n}$. 
This entails  no loss of generality, as we aim to combine the second order discretization with the cubature method, which approximates directly the law of $X$ at the times of the partition.  

For the benefit of methods that would combine the present  discretization with a Monte Carlo simulation, that most probably requires some sort of discretization
for the forward process $X$, we note that all results regarding the second order scheme (particularly Theorem \ref{mainDiscrError} and Corollary \ref{mainRate})
remain valid, once one  replaces $X$ with a second  order approximation.
A wealth of such approximations\footnote{Such approximations can be chosen to be \emph{explicit}, in other words completely derivative free.} are presented on chapters 12- 15 of \cite{KP}.
\end{rem}

Working towards a higher order discretization of the backward part of \eqref{BSDE1} we shall rely heavily on the Stratonovich-Taylor expansions. Hence we need 
to fix notation and present some elementary facts regarding these expansions:

\subsection{Stratonovich-Taylor expansions}
In the following we will use the set of multi-indices
$\mathcal{A}=\{\emptyset \}\cup \bigcup_{m=1}^{\infty }\{0,1,\ldots ,d\}^{m}$ endowed with the norms $|\cdot|$ and $\left\|\cdot\right\|$ given by
\[
|\beta |= \mbox{ length of } \beta,\quad \norm{\beta} :=|\beta |+\mathrm{Card}\{j:\beta _{j}=0,\,\,1\leq j\leq |\beta |\}.
\]
Clearly $|\emptyset|=\|\emptyset\| =0$. For a $\beta =(j_1,\ldots,j_l) \in \Ac$ we also write $\,\beta -=(j_{1},\ldots ,j_{l-1})$ and $\,-\beta =(j_{2},\ldots ,j_{l}).$ 
We define the subsets of $\mathcal{A},$ \begin{equation*}
\mathcal{A}_{m}=\{\beta \in \mathcal{A}:\,\,\norm{ \beta} \leq m\}\;\;\,%
\mathrm{and}\,\,\ \mathcal{A}_{m}^{1}=\{\beta \in \mathcal{A}\backslash
\{\emptyset ,(0)\}:\,\,\norm{ \beta} \leq m\}.
\end{equation*}
Given two multi-indices $\alpha =(\alpha _{1},\ldots ,\alpha _{k})$ and $%
\beta =(\beta _{1},\ldots ,\beta _{l})$ we define their concatenation as $\alpha \ast \beta =(\alpha _{1},\ldots ,\alpha _{k},\beta _{1},\ldots ,\beta
_{l})$. For a suitably chosen function $f$ and a multi-index $\beta =(\beta
_{1},\ldots ,\beta _{l})$, we define the iterated Stratonovich/Lebesgue
integrals as follows 
\begin{equation*}
J_{\beta }[f]_{t,\,s}:= 
\begin{cases}
f(s) & |\beta |=0 \\ 
\int_{t}^{s}J_{\beta -}[f]_{t,\,u}du & l\geq 1,\,\,j_{l}=0 \\ 
\int_{t}^{s}J_{\beta -}[f]_{t,\,u}\circ dW^{j_{l}}(u) & l\geq 1,j_{l}\not=0
\end{cases}
.\end{equation*}
Let $L^j,\ j=0,\ldots,d$ be the operators:
\[
L^j := \sum_{k=1}^q V_j^k \partial_{x_k},\quad j=1,\ldots,d,\quad\quad
L^0 := \partial_t + \sum_{k=1}^q V_0^k \partial_{x_k}.
\] 
The iteration of this family of operators is understood as :
$L^{\alpha}\,f:=L^{\alpha_1}\ldots L^{\alpha_n}\,f,$ 
$\alpha=(\alpha_1,\ldots,\alpha_n)$ and we also use the convention $L^{\emptyset} f = f$.

Given a multi index $\alpha=(\alpha_1,\ldots,\alpha_n)$  the following identity is proven in Proposition 5.2.10 of \cite{KP}
\begin{equation}
\label{Strat1}
W_t^j\, J\,^{\alpha}[1]_{0,t} =\sum_{k=0}^{n+1} J_{0,t}\,^{(\alpha_1,\ldots,\alpha_{k},j,\alpha_{k+1},\alpha_n)},\quad j=1,\ldots,d.
\end{equation}
A direct consequence of \eqref{Strat1} (Corollary 5.2.11 of \cite{KP}) is that if $\alpha = (k,k,\ldots,k),\,\,k=0,\ldots,d$
with $|\alpha|=m$, then 
\begin{equation}
\label{Strat2}
J^{\alpha}[1]_{0,t} = \frac{1}{m!}\left(J\,^{(k)}[1]_{0,t}\right)^k
\end{equation}

The (conditional) expectations of iterated integrals can also be computed. In particular we have a characterization for those 
iterated integrals that have non zero expectation (see, for example, \cite{CG}):
\begin{lemma}
\label{StratExpectations} Let $\alpha =(i_{1},\ldots ,i_{r})$ be an arbitrary multi-index with $\norm{\alpha}=m$ 
and $t\in \lbrack 0,T]$. If $m$ is odd,
then $\mathbb{E}[J_{\alpha }[1]_{0,t}]=0$ and if $m$ is even then 
\begin{equation*}
\mathbb{E}[J_{\alpha }[1]_{0,t}]= 
\begin{cases}
\frac{t^{m/2}}{2^{r-m/2}(m/2)!} & \mathrm{if}\;\alpha \in \mathcal{A}_{m,r}
\\ 
0 & \mathrm{otherwise}
\end{cases}
,
\end{equation*}
where$\;A_{m,r}\subset \mathcal{A}_{m}$ is the set of multi-indices with $%
\alpha =\alpha _{1}\ast \ldots \ast \alpha _{m/2}\in \mathcal{A}_{m}$, such
that $\alpha _{i}=(0)$ or $\alpha _{i}=(j,j),\,j\in \{1,\ldots ,d\}.\;$
\end{lemma}

A non empty subset $\mathcal{G}\subseteq \mathcal{A}$ is called a
hierarchical set if 
$\sup_{\alpha \in \mathcal{G}}|\alpha |<\infty$ and  
$-\alpha \in\mathcal{G}\text{ for any }\alpha \in \mathcal{G}\backslash \{\emptyset \}.
$
Given a hierarchical set $\mathcal{G}$ we define the remainder set 
\[
\mathcal{B}(\mathcal{G})=\{\beta \in \mathcal{A}\backslash \mathcal{G}:\,\,-\beta \in \mathcal{G}\}.
\]
Note that $\mathcal{A}_{m}$ and $\mathcal{A}_{m}^{1}$ are
hierarchical sets. If $f:\mathbb{R}^{d}\rightarrow \mathbb{R}$ is a smooth
function with all partial derivatives bounded, then by repeatedly applying
the It\^{o}-Stratonovich formula, we obtain the Stratonovich-Taylor expansion 
\begin{equation}
\begin{split}
f(X_{t}^{0,x})&=\sum_{\alpha \in \mathcal{K}}L^{\alpha }f(x)J^{\alpha
}[1]_{0,t}+\sum_{\alpha \in \mathcal{B}(\mathcal{K})}J^{\alpha }[L^{\alpha
}f(\cdot )] _{0,t}\\
&=: \sum_{\alpha \in \mathcal{K}}L^{\alpha }f(x)J^{\alpha}[1]_{0,t} + R_m(t,x,f)
\end{split}
\label{Taylor}
\end{equation}
for any hierarchical set $\mathcal{K}.$ The second term on the right hand side of (\ref{Taylor}) is called the
m-th order remainder process and is denoted by $R_{m}(t,x,f).$ Obviously as $m$ increases, and for small $t$,  the remainder process  
gets  smaller and smaller. 

A hierarchical set that facilitates computations for Stratonovich-Taylor expansions is $\Ac(m)$.
For this set we have that 
\[
\mathcal{B}\left(\Ac(m)\right)  =\{(j)\star\beta|j=0,\ldots,d,\quad \|\beta\|= m\}.
\]
To estimate the remainder that corresponds to $\mathcal{B}\left(\Ac(m)\right)$, we need to estimate iterated Stratonovich integrals 
of the form $J^{\alpha}\left[g\paren{\cdot,X_{\cdot}}\right]_{s,t}$ for appropriate function $g$ such that the previous integral makes sense.
For  $0\le s \le t$ with $t-s<1$, we have 
\begin{equation}
\label{TaylorRemainder1}
\begin{split}
\left|\sup_{s\le r \le t} \ex_s\left[J^{\alpha}\left[g(\cdot,X_{\cdot})\right]_{s,r}\right]\right| &\le C(t-s)^{\norm{\alpha}/2} 
\Bigl(\sup_{s\le r \le t}\bigl\{\ex_s\left[\left|L^{\alpha_n}g(r,X_r)\right|1_{\alpha_n\neq 0}\right]\\
&\hspace{40mm} + \ex_s\left[\left|g(r,X_r)\right|1_{\alpha_n= 0}\right]\bigr\}\Bigr)
\end{split},
\end{equation}
see chapter 5 of \citet{KP}.
Applying the estimate \eqref{TaylorRemainder1} to every index in the set  $\mathcal{B}\left(\Ac(m)\right)$ , provides us with the following estimate for the remainder 
of the Taylor formula,    
\begin{equation}
\ex\brac{\abs{R_m(t,X,f)}^p}^{1/p} \le  C(t-s)^{(m+1)/2} \max_{m\le \|\gamma\| \le m+2} \sup_{0\le r\le t}\ex\brac{\abs{L^{\gamma} f(r,X_r)}^p}^{1/p}.
\label{TaylorRemainder2}
\end{equation}

To conclude our preliminary results, we shortly discuss the backward PDE \eqref{b_pde}. For the latter, when all coefficients are smooth and bounded, it is 
well understood that it posseses a smooth and bounded solution. Moreover, a smooth solution $u(t,x)$ will still exist when working only under 
\textbf{(C1)} for $t\in [0,T)$. The solution itself will be bounded on all $[0,T]$. Its derivatives however shall not. The next result
characterises the behaviour of the derivatives.

\begin{thm}[\citet{delarue}]
\label{pdeDer}
Let \textbf{(A)} and \textbf{(B)} hold true and assume further that $\Phi$ is Lipschitz continuous. Then there exists a unique $u\in C_b^{\lceil m/2 \rceil,m}\paren{[0,T)\times \reals^d}$ 
that solves \eqref{b_pde}. Moreover, for any $\alpha \in \Ac(m)$ we have
\[
\left \|D^{\alpha} u(t,\cdot)\right\|_{\infty} \le C\frac{\norm{\nabla \Phi}_{\infty}}{\paren{T-t}^{(\norm{\alpha}-1)/2}},\quad t\in [0,T),
\]
for a constant $C$ independent of $u$. 
\end{thm}
In other words, when the final condition is Lipschitz continuous (and the driver is smooth), the derivatives of the solution of \eqref{b_pde}, explode  as $t\uparrow T$ with the given rate.   
The exact rates are important to quantify the error of our method. 

\section{One step discretization}
To understand the terms that a second order scheme should incorporate, we present first the intuitive arguments that lead to the discretization scheme  \eqref{BkwdEuler}. 
We consider  the backward part of \eqref{BSDE1} between two successive times of the  partition  $\pi$
\[
Y_{t_i} = Y_{t_{i+1}} +\int_{t_i}^{t_{i+1}} f(X_s,Y_s,Z_s) ds - \int_{t_i}^{t_{i+1}} Z_s \cdot dW_s
\]
and discretize the Riemann integral using the left hand side point, as in \cite{BouchardTouzi}, thus leading to an implicit equation for $Y_{t_i}$ and the stochastic part in the usual way, to obtain
\begin{equation}
\label{euler1}
Y_{t_i} \simeq Y_{t_{i+1}} +\dt{i+1} f(X_{t_i},Y_{t_i},Z_{t_i})  -Z_{t_i} \cdot \Delta W_{i+1}.
\end{equation}
By conditioning  \eqref{euler1} with respect to  $\Fc_{t_i}$ in \eqref{euler1} we obtain a first order  approximation for $Y_{t_i}$   
\begin{equation}
\label{euler2}
Y_{t_i} \simeq \ex\brac{Y_{t_{i+1}}\Bigr|\Fc_{t_i}} +\dt{i+1} f(X_{t_i},Y_{t_i},Z_{t_i}), 
\end{equation}
but for the presence of $Z_{t_i}$.
To treat the $Z_{t_i}$, we can multiply both sides of \eqref{euler1} by $\dB{l},\,\,l=1,\ldots,d$ and condition with respect to $\Fc_{t_i}$, to obtain 
\begin{equation}
\label{euler3}
Z_{t_i}^l\simeq \ex\brac{Y_{t_{i+1}} \frac{\dB{l}}{\dt{i+1}}\Bigr|\Fc_{t_i}},\quad l=1,\ldots,d.
\end{equation}
Observe that scheme \eqref{BkwdEuler} is just the backward iteration of equations \eqref{euler2} and \eqref{euler3}.

Clearly,  to achieve a higher order approximation of the BSDE, we need to add more terms in its (stochastic) expansion ideally without involving any derivatives.
As far as the driver is concerned, the next elementary result tells us that the Crank-Nickolson rule indeed achieves a third order local error, 
thus leading to a second order global error: 
\begin{lemma}
\label{SimpsonError}
Let assumptions (A) and (B) hold true, $\Phi\in C_{Lip}\paren{\reals^d}$ and  denote by 
$\bar{f}(t,x)\equiv f(x,u(t,x),\nabla u(t,x)V(x))$,
where $u(t,x)$ is the classical solution to PDE \eqref{b_pde}. Then  there exists a constant $C$ independent of the partition and of $u$
such that
\[
\abs{\ex_i\brac{\int_{t_i}^{t_{i+1}}\bar{f}(s,X_s) ds -  \frac{\dt{i+1}}{2}\paren{\bar{f}(t_i,X_{t_i}) + \bar{f}(t_{i+1},X_{t_{i+1}})}}}
%\le C \abs{\ex_i\brac{\int_{t_i}^{t_{i+1}} R_3(\bar{f},t_i,s)ds}}
\le C\dt{i+1}^3 \max_{\norm{\alpha}=4,5}\left\|L^{\alpha}u(t_{i+1},\cdot)\right\|_{\infty}
.\]
\end{lemma}   
\begin{proof}
Note first that the nonlinear Feynman Kac formula tells us that 
$Y_t = u(t,X_t)$ and $Z^l_t = L^lu(t,X_t)$, $l=1,\ldots,d,$
where the existence of $u,\,\nabla u $ is guaranteed by Theorem \eqref{pdeDer}.
The proof is  a simple consequence of the expansion \eqref{Taylor}, using the hierarchical set $\Ac(3)$, applied to the integrand of the integral 
\begin{equation}
\label{SimpsonProof1}
\begin{split}
&\ex_i\brac{\int_{t_i}^{t_{i+1}}\bar{f}(s,X_s)ds} =\\ 
&\qquad \dt{i+1}\,\bar{f}(t_i,X_i)+ \ex_{i}\left[\int_{t_i}^{t_{i+1}} \,\left( \sum_{\alpha \in \Ac_0(3)} L^{\alpha} \bar{f}(t_i,X_i)\, J^{\alpha}[1]_{t_i,\,u} +\sum_{\alpha \in \Ac_(5)\char92 \Ac_0(3)}
 J^{\alpha}\,[L^{\alpha} \bar{f}(\cdot,X_{\cdot})]_{t_i,\,u}\right)\,\rm{d}u\right]
\end{split}
\end{equation}
and to $\bar{f}\paren{t_{i+1},X_{i+1}}$,
\begin{equation}
\label{SimpsonProof2}
\begin{split}
&\dHalf{i+1}\paren{f\paren{X_i,Y_i,Z_i} +\ex_i\brac{f\paren{X_{i+1},Y_{i+1},Z_{i+1}}}} = \\
&\qquad \dt{i+1}\,\bar{f}(t_i,X_i) + \dHalf{i+1}\ex_i\left(\sum_{\alpha \in \Ac_0(3)} L^{\alpha} \bar{f}(t_i,X_i)\, J^{\alpha}[1]_{t_i,\,t_{i+1}}
 + \sum_{\alpha \in \Ac_0(5)\char92 \Ac_0(3)}  J^{\alpha}[L^{\alpha} \bar{f}(\cdot,X_{\cdot})]_{t_i,\,t_{i+1}} \right)
\end{split}
\end{equation}
 If $\alpha \in \Ac_0(3)$, then one of the following holds 
 \[
 \alpha = \left\{\begin{array}{l l} (i),& i=0,\ldots,d\\
 (0,i),\,(i,0),&i=1,\ldots,d\\
 (i,j),\,(i,j,k)\,\,&i,j,k=1,\ldots,d,
 \end{array}\right.
 \]
 Elementary facts about stochastic integrals and Lemma  \ref{StratExpectations} tell us that 
 $\ex_s\left[J^{\alpha}[1]_{s,\,u}\right] =0$, $u \ge s$ for all $\alpha$'s as above except when $\alpha =(0),\,\mbox{ or }(j,j),\,j=1,\ldots,d.$
 For these two cases, we have
 \[
 \ex_s\left[J^{(0)}[1]_{s,\,u}\right] = u - s\ \ \ \
\mathrm{and} \ \ \  
\ex_s\left[J^{(j,j)}[1]_{s,\,u}\right] = \frac{1}{2}\Ex{ \left(W^j_u-W^j_{s}\right)^2} = \frac{u-s}{2}.
\]
Substituting the above into \eqref{SimpsonProof1}, \eqref{SimpsonProof2}   and taking their difference we have
 \[
 \begin{split}
&\abs{\ex_i\brac{\int_{t_i}^{t_{i+1}}\bar{f}(s,X_s) ds -  \frac{\dt{i+1}}{2}\paren{\bar{f}(t_i,X_{t_i}) + \bar{f}(t_{i+1},X_{t_{i+1}})}}}\\
&\qquad \le \abs{\ex_i\left[\int_{t_i}^{t_{i+1}} \left( \sum_{\alpha \in \Ac_0(5)\char92 \Ac_0(3)}  J^{\alpha}\,[L^{\alpha} \bar{f}(\cdot,X_{\cdot})]_{t_i,\,u}\right)\,\rm{d}u
            - \dHalf{i+1} \sum_{\alpha \in \Ac_0(5)\char92 \Ac_0(3)}  J^{\alpha}[L^{\alpha} \bar{f}(\cdot,X_{\cdot})]_{t_i,\,t_{i+1}}\right]}.
 \end{split}
 \]  
 The estimates on the remainder process complete the proof. 
 \end{proof}  

The above result combined with a Stratonovich Taylor expansion applied on $Z$ (as the integrand of the stochastic part) give us a first intuitive approach for the second order discretization: 
\begin{equation}
\label{2order1}
\begin{split}
Y_{t_i} &\simeq Y_{t_{i+1}} + \frac{\dt{i+1}}{2}\paren{f(X_{t_i},Y_{t_i},Z_{t_i}) + f(X_{t_{i+1}},Y_{t_{i+1}},Z_{t_{i+1}})}\\   
&\quad-\sum_{l=1}^d\Bigl\{ L^l u\paren{t_i,X_{t_i}}\dB{l} + \sum_{k=1}^d L^{(k,l)}u\paren{t_i,X_{t_i}} \int_{t_i}^{t_{i+1}} J^{(k)}[1]_{t_i,s}dW^l_s\\
&\hspace{20mm} +\sum_{k,j=1}^d L^{(k,j,l)}u\paren{t_i,X_{t_i}} \int_{t_i}^{t_{i+1}} J^{(k,j)}[1]_{t_i,s}dW^l_s\\
&\hspace{20mm} +  L^{(0,l)}u\paren{t_i,X_{t_i}} \int_{t_i}^{t_{i+1}} s\,dW^l_s  + \sum_{\|\alpha\|=3}L^{\alpha\star(l)}u\paren{t_i,X_{t_i}}\int_{t_i}^{t_{i+1}} J^{\alpha}[1]_{t_i,s}dW^l_s\\
&\hspace{35mm} + \int_{t_i}^{t_{i+1}} R_3(L^l u,t_i,s) dW^l_s\Bigr\}
\end{split}
\end{equation}
As in \eqref{euler3}, we need to innovate a way to recover $Z_i^l,\,l=1,\ldots,d,\,i=0,\ldots,n-1$, from \eqref{2order1}, but this time up to a second order error.

If we multiply  both sides of \eqref{2order1} by $\frac{\dB{l}}{\dt{i+1}}$ and condition with respect to $\Fc_{i}$, we shall obtain $Z^l_{t_i}$ but some surviving terms of order $\dt{i+1}$  will render the approximation
first order. For example consider 
\[
\frac{1}{\dt{i+1}}\ex\brac{L^{(0,l)}u\paren{t_i,X_{t_i}} \int_{t_i}^{t_{i+1}} s\,dW^l_s \dB{l}| \Fc_{t_i}} =  \frac{\dt{i+1}}{2}L^{(0,l)}u\paren{t_i,X_{t_i}}.
\]
Hence, we need to find an appropriate weight, that will multiply \eqref{2order1} and after conditioning  with respect to $\Fc_{t_i}$, provides a second order approximation 
for $Z_{t_i}$ by \textit{canceling out all first order terms}.  
We make the following judicious choice 
\begin{equation}
\label{2order2}
\Zc_i^l:= \lambda_1\frac{\dB{l}}{\dt{i+1}} + \lambda_2\frac{J^{(0,l)}[1]_{t_i,t_{i+1}}}{\dt{i+1}^2},\quad l=1,\ldots,d,\quad i=0,\ldots,n-1.
\end{equation}  
With a few straightforward computations, we obtain for any $q=1,\dots,d$
\[
\begin{split}
&\ex\brac{\sum_{l=1}^d L^l u\paren{t_i,X_{t_i}}\dB{l}\Zc_i^q\Bigl|\Fc_{t_i}} = \paren{\lambda_1 + \frac{\lambda_2}{2}} L^q u\paren{t_i,X_{t_i}} \\
&\ex\brac{\sum_{l=1}^d L^{(0,l)}u\paren{t_i,X_{t_i}} \int_{t_i}^{t_{i+1}} s\,dW^l_s \Zc_i^q\Bigl|\Fc_{t_i}} = \paren{\frac{\lambda_1\dt{i+1}}{2} + \frac{\lambda_2\dt{i+1}}{3}} L^{(0,q)} u\paren{t_i,X_{t_i}} \\
&\ex\brac{\sum_{l=1}^d\sum_{k,j=1}^d L^{(k,j,l)}u\paren{t_i,X_{t_i}} \int_{t_i}^{t_{i+1}} J^{(k,j)}[1]_{t_i,s}dW^l_s \Zc_i^q\Bigl|\Fc_{t_i}} 
= \paren{\frac{\lambda_1\dt{i+1}}{4} + \frac{\lambda_2\dt{i+1}}{6}} \sum_{k=1}^d L^{(k,k,q)} u\paren{t_i,X_{t_i}} \\
&\ex\brac{\sum_{l=1}^d\paren{\sum_{k=1}^d  \int_{t_i}^{t_{i+1}} J^{(k)}[1]_{t_i,s}dW^l_s+
\sum_{\|\alpha\|=3}\int_{t_i}^{t_{i+1}} J^{\alpha}[1]_{t_i,s}dW^l_s }\Zc_i^q\Bigl|\Fc_{t_i}} =0.
\end{split}
\]
By choosing $\lambda_1 =4,\,\lambda_2 = -6$, we have the following : 
\begin{lemma}
\label{zError}
Let assumptions (A), (B) hold true and let $u(t,x)$ denote the classical solution of PDE \eqref{b_pde}. Set 
\[
\Zc_i^l:= 4\frac{\dB{l}}{\dt{i+1}}  - 6\frac{J^{(0,l)}[1]_{t_i,t_{i+1}}}{\dt{i+1}^2},\quad l=1,\ldots,d,\quad i= 0,\ldots,n-2.
\]
Then
\[
\abs{Z_{t_i}^l - \ex_i\brac{\paren{Y_{t_{i+1}} + \frac{\dt{i+1}}{2} f(X_{t_{i+1}},Y_{t_{i+1}},Z_{t_{i+1}})} \Zc_i^l}}
\le \dt{i+1}^2 \max_{\norm{\alpha}=4,5} \norm{L^{\alpha}u(t_{i+1},\cdot)}_{\infty}. 
\]
\end{lemma} 
\begin{proof}
It should be clear from the discussion preceding Lemma \ref{zError} combined with the result of  Lemma \ref{SimpsonError}
that 
\[
\begin{split}
&\abs{Z_{t_i}^l - \ex_i\brac{\paren{Y_{t_{i+1}} + \frac{\dt{i+1}}{2} f(X_{t_{i+1}},Y_{t_{i+1}},Z_{t_{i+1}})} \Zc_i^l}}\\ 
&\quad \le \abs{\ex_i\brac{\Zc_i^l\paren{\int_{t_i}^{t_{i+1}} R_3(\bar{f},t_i,s) \,ds+\int_{t_i}^{t_{i+1}} R_3(L^lu,t_i,s)dW^l_s}}}
\end{split}
\]
The result then follows from the estimate \eqref{TaylorRemainder2} on the  remainder process and the Cauchy Schwartz inequality.
\end{proof}
Lemmas \ref{SimpsonError} and \ref{zError} dictate the following second order algorithm :

\vspace{3mm}
$\bullet$\textbf{Initialization}

\textbf{ If (C1) is in force :}
\[
\Ypi{n}: = \Phi(X_n), \quad \Zpi{n}  := 0,\mbox{ and  }
\quad\Zpi{n-1} := \zEuler{n-1},\quad \Ypi{n-1}:=\yEuler{n-1}.
\]
\textbf{If (C2) is in force:} 
\[
\Ypi{n}: = \Phi(X_n),\quad \Zpi{n}:= \nabla\Phi(X_n)V(X_n)
\]

$\bullet$\textbf{Backward induction: } 
\begin{equation}
\label{scheme2}
\begin{split}
&\Zpi{i} = \ex_i\brac{\paren{\Ypi{i+1} +\dHalf{i+1}f(X_{i+1},\Ypi{i+1},\Zpi{i+1})}\Zc_{i}},\,\, \Zc_{i}:=(\Zc_{i}^1,\ldots,\Zc_{i}^d)^T\\
&\Ypi{i} =  \ex_i\brac{\Ypi{i+1}}   
 +\frac{\delta_{i+1}}{2}\Bigl(f\left(X_i,\Ypi{i},\Zpi{i}\right)+\ex_{i}\left[f\left(X_{i+1},\Ypi{i+1},\Zpi{i+1}\right)\right]\Bigr)
\end{split}.
\end{equation}

A small clarification is perhaps in order for the peculiarity of our scheme in the first backward step. If the terminal condition is not smooth, then none of the intuitive arguments that we have presented 
can apply at time $t_n=T$. Hence, when working under \textbf{(C1)},
 we take the first backward step in an Euler fashion. After this, we expect the PDE to ``smooth out''
the value function and hence all that we have discussed so far apply. We will see in Corollary \ref{mainRate} how a non equidistant partition can compensate for this.

\begin{thm}
\label{mainDiscrError}
Let assumptions \textbf{(A),(B)}  and either of \textbf{(C1), (C2)} hold true. 
Then there exists a constant  $C>0$ such that 
\[ 
\begin{split}
&\max_{0\le i\le n-2} \brac{\abs{Y_{t_i} - \Ypi{i}}^2 +\frac{\dt{i+1}}{4d}\abs{Z_{t_i} - \Zpi{i}}^2}\\
&\qquad \le C\paren{\abs{Y_{t_{n-1}} - \Ypi{n-1}}^2 +\frac{\dt{n}}{4d}\abs{Z_{t_{n-1}} - \Zpi{n-1}}^2} +\sum_{i=1}^{n-1} \dt{i}^5 
\max_{\norm{\alpha}=4,5} \norm{ L^{\alpha}u(t_{i},\cdot)}^2_{\infty}
\end{split}
\]
\end{thm} 
\begin{proof}
In the following proof, $C$ will denote a constant whose value might change from line to line. It will however be independent
of the partition  and of the bounds of the derivatives of the solution of \eqref{b_pde}. 
For ease of notation, we set 
\[
\begin{split}
&\dY{i}:= Y_{t_i} - \Ypi{i},\quad \dZ{i}:= Z_{t_i} - \Zpi{i},\\
&\df{i}:= f\paren{X_{t_i},Y_{t_i},Z_{t_i}} - f\paren{X_{t_i},\Ypi{i},\Zpi{i}}\\
&\Psi_{i+1} := Y_{t_{i+1}} +\dHalf{i+1}f(X_{t_{i+1}},Y_{t_{i+1}},Z_{t_{i+1}}),\quad
 \Psi^{\pi}_{i+1} := \Ypi{i+1} +\dHalf{i+1}f(X_{t_{i+1}},\Ypi{i+1},\Zpi{i+1})\\
&\Delta \Psi_{i+1} = \Psi_{i+1} -\Psi_{i+1}^{\pi}.
\end{split}
\]
Let us fix a value for $i=0,\ldots,n-2$.
We consider the difference of the solution of the BSDE at time $t_i$ and of scheme \eqref{scheme2}:
\begin{equation}
\begin{split}
\dY{i} &=\ex_{i}\left[\dY{i+1}\right] 
 +\frac{\delta_{i+1}}{2} \ex_{i}\left[f\left(X_i,\Ypi{i},\Zpi{i}\right)+f\left(X_{i+1},\Ypi{i+1},\Zpi{i+1}\right) \right]\\
& \qquad \mp\frac{\delta_{i+1}}{2} \ex_{i}\left[\bar{f}\left(t_i,X_{i}\right)+\bar{f}\left(t_{i+1},X_{i+1}\right)\right]
 -  \olo\ex_{i}\left[\bar{f}\left(s,X_s\right)\right]\,\rm{d} s
\end{split}
\label{scheme1_1}
\end{equation}
According to the estimates of Lemma \ref{SimpsonError} we have that
\begin{equation}
\label{scheme1_2}
\begin{split}
&\abs{ \ex_{i}\left[f\left(X_i,\Ypi{i},\Zpi{i}\right)+f\left(X_{i+1},\Ypi{i+1},\Zpi{i+1}\right) \right]
 -\frac{\delta_{i+1}}{2} \ex_{i}\left[\bar{f}\left(t_i,X_{i}\right)+\bar{f}\left(t_{i+1},X_{i+1}\right)\right]}\\
&\qquad \le C\dt{i+1}^3 \max_{\norm{\alpha}=4,5} \norm{L^{\alpha}u(t_{i+1},\cdot)}_{\infty}
\end{split}
\end{equation}
    
Moreover, according to the mean value theorem, we can argue on the existence of  a real number and vector $\mu_1\in\reals,\,\nu_1 \in \reals^d$ 
bounded by $K$, such that 
\begin{equation}
\label{scheme1_3}
\frac{\delta_{i+1}}{2} \df{i}= \dHalf{i+1}\left( \mu_1\dY{i} + \nu_1\cdot\dZ{i}\right)
\end{equation} 
Combining \eqref{scheme1_1}-\eqref{scheme1_3} with Young's inequality with $\gamma_1>0$, we have 
\begin{equation}
\label{scheme1_4} 
\begin{split}
\abs{\dY{i}}^2&\le \paren{1 + \gamma_1\dt{i+1}}\abs{\ex_i\brac{\Delta^{\pi}_{i+1} \Psi}}^2+\paren{1+\frac{1}{ \gamma_1\dt{i+1}}}C\dt{i+1}^2\paren{\abs{\dY{i}}^2 + \abs{\dZ{i}}^2 }\\
&\qquad + \paren{1+\frac{1}{\gamma_1 \dt{i+1}}}C\dt{i+1}^6\max_{\norm{\alpha}=4,5} \norm{L^{\alpha}u(t_{i+1},\cdot)}_{\infty}^2
\end{split}
\end{equation}
Next, observe that for any random variable $F$ which is measurable with respect to $\Fc_{t_{i+1}}$ we have
\[
\abs{\ex_i\brac{F\Zc_i}}^2 = \abs{\ex_i\brac{\paren{F-\ex_i\brac{F}}\Zc_i}}^2\le \frac{1}{\dt{i+1}}\paren{\ex_i\brac{F^2} - \ex_i\brac{F}^2}
\]
Combining this with definition of $\Zpi{i},\,i=0,\ldots,n-2$ and the conclusion of Lemma  \ref{zError}, we have
\begin{equation}
\label{scheme1_5}
\dt{i+1}\ex\brac{\abs{\dZ{i}}^2}
 \le 2d\paren{\ex\brac{\abs{\Delta \Psi_{i+1}}^2}  - \ex\brac{\abs{\ex_i\brac{\Delta \Psi_{i+1}}}^2}} 
+C \dt{i+1}^6\max_{\norm{\alpha}=4,5} \norm{L^{\alpha} u(t_{i+1},\cdot)}^2_{\infty} 
\end{equation}  
Putting together \eqref{scheme1_4} and \eqref{scheme1_5}  we get
\begin{equation}
\label{scheme1_6}
\begin{split}
&\ex\brac{\abs{\dY{i}}^2 + \frac{\dt{i+1}}{4d}\abs{\dZ{i}}^2}\\
& \quad\le \paren{ 1 + \gamma_1 \dt{i+1}}\ex\abs{\ex_i\brac{\Delta^{\pi}_{i+1}\Psi}}^2\\  
&\qquad + C\dt{i+1}\ex\brac{\abs{\dY{i}}^2} + \paren{\frac{C}{\gamma_1} + \frac{1}{4d} +C \dt{i+1}}\dt{i+1} \ex\brac{\abs{\dZ{i}}^2}  \\
&\qquad +C\dt{i+1}^5 \max_{\norm{\alpha}=4,5} \norm{L^{\alpha}u(t_{i+1},\cdot)}^2_{\infty}\\
&\quad \le  \paren{ 1 + \gamma_1 \dt{i+1}}\ex\brac{\abs{\Delta^{\pi}_{i+1}\Psi}^2}\\
&\qquad C\dt{i+1}  \ex\brac{\abs{\dY{i}}^2} + C\dt{i+1}^5 \max_{\norm{\alpha}=4,5} \norm{L^{\alpha}u(t_{i+1},\cdot)}^2_{\infty}
\end{split}
\end{equation}   
where we have chosen $\gamma_1 = C4d$. 

We can argue once more with the mean value theorem and Young's inequality to deduce that
\[
\begin{split}
&\paren{1-C\dt{i+1}}\ex\brac{\abs{\dY{i}}^2 + \frac{\dt{i+1}}{4d}\abs{\dZ{i}}^2} \\
&\qquad \le (1 + C^{\prime}\dt{i+1})\paren{\ex\brac{\abs{\dY{i+1}}^2 + \frac{\dt{i+1}}{4d}\abs{\dZ{i+1}}^2}  } 
+C\dt{i+1}^5 \max_{\norm{\alpha}=4,5} \norm{L^{\alpha}u(t_{i+1},\cdot)}^2_{\infty}.
\end{split}
\]
for some different constant $C^{\prime}$.
By appealing to the discrete version of Gronwall's lemma we complete the proof. 
\end{proof}
Theorem \ref{mainDiscrError} justifies the fact that scheme \eqref{scheme2} produces a  second order discretization of the backward component of the BSDE when \textbf{(A), (B)} hold true 
and $\Phi$ is smooth. Clearly  in this case, 
$\max_{\norm{\alpha}=4,5} \norm{L^{\alpha}u(t_{i+1},\cdot)}^2_{\infty}$ is uniformly bounded, the error on the first backward step is of $\Oc(1/n^2)$ when $n$ points are used on the time discretization and hence, with a uniform partition $n$ points, we obtain an overall error of $\Oc(1/n^2).$ 

The more interesting case ocurs when $\Phi$ is only Lipschitz continuous. 
Given the results of Theorem \ref{pdeDer},  we expect the bound on the derivatives of $u$ to explode as $ t\uparrow T$. 
To compensate for this and for the fact that our first backward step is of Euler style, we work with a non equidistant partition that becomes more dense as we approach $T$. Hence 
we are still able to achieve a $1/n^2$ rate of convergence with a $n$-points partition.

\begin{cor}
\label{mainRate}
Let assumptions \textbf{(A), (B)} hold true and assume that $\Phi$ is Lipschitz continuous. Consider the discretization \eqref{scheme2} along the partition
$\pi$ : 
\[
t_i = T\paren{1 - \paren{1-\frac{i}{n}}^{\beta}},\quad i=0,\ldots,n,\quad \beta \ge 5.
\]
Then, there exists a constant $C$ independent of the partition and of the value function $u$, such that 
\[
\max_{0\le i\le n-1} \ex\brac{\abs{Y_{t_i} - \Ypi{i}}^2 + \abs{Z_{t_i} - \Zpi{i}}^2 }^{1/2}
\le \frac{C}{n^2}.
\]
\end{cor}
\begin{proof}
Clearly, we can establish the result by estimating the terms $\dt{i+1}^5 \max_{\norm{\alpha}=4,5} \norm{L^{\alpha}u(t_{i+1},\cdot)}^2_{\infty}$. 
According to Theorem \ref{pdeDer}, we have
\begin{equation}
\label{mainRate1}
\begin{split}
\sum_{i=1}^{n-1} \dt{i}^5  \max_{\norm{\alpha}=4,5} \norm{L^{\alpha}u(t_{i},\cdot)}^2_{\infty} 
&\le \sum_{i=1}^{n-1} \dt{i}^5 \frac{C\norm{\nabla \Phi}_{\infty}}{(T-t_i)^{4}}\\
& = \sum_{i=1}^{n-1} T^{5}\paren{\int_{1-\frac{i}{n}}^{1-\frac{i-1}{n}} \beta s^{\beta-1}\,ds}^5\frac{C\,\norm{\nabla \Phi}_{\infty}}{T^4\paren{1-\frac{i}{n}}^{4\beta}}\\  
&\le \sum_{i=1}^{n-1} \frac{T \beta^5}{n^5} \frac{\paren{1-\frac{i-1}{n}}^{5(\beta-1)}}{\paren{1-\frac{i}{n}}^{4\beta}}\le C/n^4
\end{split}
\end{equation} 
since $\beta \ge 5$.  

To complete our proof we estimate the error on the first backward step. Using the standard estimate on the Euler discretization error of Lebesque integrals
we have
\[
\begin{split}
\abs{Y_{n-1} -\yEuler{n-1}}^2  &=\left|\ex_i\brac{\int_{t_{n-1}}^{t_n} \bar{f}\paren{s,X_s} ds} \mp f\paren{X_{n-1},Y_{n-1},Z_{n-1}}\dt{n}\right.\\
&\hspace{30mm}\left. - f\paren{X_{n-1},\yEuler{n-1},\zEuler{n-1}}\dt{n}\right|^2\\
&\quad \le C\dt{n}^2\paren{1 + \abs{Y_{n-1} -\yEuler{n-1}}^2+ \abs{Z_{n-1} -\zEuler{n-1}}^2}.
\end{split}
 \]
 where once again, we have used the mean value theorem.
Rearranging terms, we may argue on the existence of a constant $C$ such that
\begin{equation}
\label{mainRate2} 
\begin{split}
&\paren{1-C\dt{n}}\paren{\abs{Y_{n-1} -\yEuler{n-1}}^2 +\frac{\dt{n}}{4d}\abs{Z_{n-1}-\zEuler{n-1}}^2 }\\
&\hspace{10mm}\le C\dt{n}^2 + C\dt{n}\abs{Z_{n-1} -\zEuler{n-1}}^2
\end{split}
\end{equation}
From standard estimates on BSDEs we know that under \textbf{(A), (B)} and \textbf{(C1)}
\[
\sup_{0\le t \le T} \ex\brac{\abs{Z_t}^2} < +\infty.
\]
Lastly, under \textbf{(C1)} we have that 
\begin{equation}
\label{mainRate3}
\begin{split}
\ex\brac{\abs{\zEuler{n-1}}^2} &= \frac{1}{\dt{n}^2}\ex\brac{\abs{\ex_{n-1}\brac{\Phi(X_n) \Delta W_n}}^2}\\
&= \frac{1}{\dt{n}^2}\ex\brac{\abs{ \ex_{n-1}\brac{\paren{\Phi(X_n)-\Phi(X_{n-1})} \Delta W_n}}^2}
\le C
\end{split}
\end{equation}
Substituting \eqref{mainRate3} into \eqref{mainRate2} and then taking square roots,
completes the proof.
\end{proof}

\bibliographystyle{plainnat}
\bibliography{../myBiblio}

\begin{thebibliography}{16}
\providecommand{\natexlab}[1]{#1}
\providecommand{\url}[1]{\texttt{#1}}
\expandafter\ifx\csname urlstyle\endcsname\relax
  \providecommand{\doi}[1]{doi: #1}\else
  \providecommand{\doi}{doi: \begingroup \urlstyle{rm}\Url}\fi

\bibitem[Bally and Pag\`es(2003)]{ballypages}
Vlad Bally and Gilles Pag\`es.
\newblock Error analysis of the quantization algorithm for obstacle problems.
\newblock \emph{Stochastic Processes and their Applications}, pages 1--40,
  2003.

\bibitem[Bender and Zhang(2008)]{benderzhang}
Christian Bender and Jianfeng Zhang.
\newblock Time discretization and {M}arkovian iteration for coupled
  {F}{B}{S}{D}{E}s.
\newblock \emph{Annals of Applied Probability}, 18\penalty0 (1):\penalty0
  143--177, 2008.

\bibitem[Bouchard and Touzi(2004)]{BouchardTouzi}
Bruno Bouchard and Nizar Touzi.
\newblock Discrete time approximation and {M}onte {C}arlo simulation for
  {B}ackward {S}tochastic {D}ifferential {E}quations.
\newblock \emph{Stochastic processes and their applications}, 111:\penalty0
  175--206, 2004.

\bibitem[Crisan and Ghazali(2007)]{CG}
Dan Crisan and Saadia Ghazali.
\newblock \emph{On the convergence rate of a general class of {S}{D}{E}s},
  pages 221--248.
\newblock World Scientific Pubishing, NJ, 2007.

\bibitem[Crisan and Manolarakis(2007)]{CriMan}
Dan Crisan and Konstantinos Manolarakis.
\newblock Numerical solution for a {B}{S}{D}{E} using the {C}ubature method.
\newblock preprint available at http://www2.imperial.ac.uk/~dcrisan/, 2007.

\bibitem[Crisan and Manolarakis(2010)]{m2an}
Dan Crisan and Konstantinos Manolarakis.
\newblock Probabilistic methods for semilinear partial differential equations.
  {A}pplications to finance.
\newblock \emph{ESAIM: Mathematical Modelling and Numerical Analysis},
  44:\penalty0 1107--1133, 2010.

\bibitem[Delarue(2010)]{delarue}
Francois Delarue.
\newblock Smoothing property of a non -degenerate semi linear {P}{D}{E}.
\newblock 2010.
\newblock preprint.

\bibitem[Gobet and Labart(2007)]{gobetlabart}
Emanuel Gobet and Celine Labart.
\newblock Error expansion for the discretization of {B}ackward {S}tochastic
  {D}ifferential {E}quations.
\newblock \emph{Stochastic processes and their applications}, 117\penalty0
  (7):\penalty0 803--829, 2007.

\bibitem[Gobet et~al.(2005)Gobet, Lemor, and Warin]{gobetlemorwarin}
Emanuel Gobet, {Jan Phillip} Lemor, and Xavier Warin.
\newblock A regression based {M}onte {C}arlo method to solve {B}ackward
  {S}tochastic {D}ifferential {E}quations.
\newblock \emph{Annals of Applied Probability}, 15\penalty0 (3):\penalty0
  2172--2002, 2005.

\bibitem[Gobet et~al.(2006)Gobet, Lemor, and Warin]{gobetlemorwarin2}
Emanuel Gobet, {Jan Phillip} Lemor, and Xavier Warin.
\newblock Rate of convergence of an empirical regression method for solving
  generalized backward stochastic differential equations.
\newblock \emph{Bernoulli}, 12\penalty0 (5):\penalty0 889--916, 2006.

\bibitem[Kloeden and Platen(1999)]{KP}
Peter Kloeden and Eckhard Platen.
\newblock \emph{Numerical solutions of {S}tochastic {D}ifferential
  {E}quations}.
\newblock Springer, 1999.

\bibitem[Kusuoka(2003)]{Kusuoka}
Shigeo Kusuoka.
\newblock Approximation of expectations of diffusion processes based on {L}ie
  algebra and {M}alliavin calculus.
\newblock \emph{UTMS}, 34, 2003.

\bibitem[Lyons and Victoir(2004)]{LV}
Terry Lyons and Nicolas Victoir.
\newblock Cubature on {W}iener space.
\newblock \emph{Proc. Royal Soc. London}, 468:\penalty0 169--198, 2004.

\bibitem[Ma and Zhang(2002)]{Mazhang}
Jin Ma and Jianfeng Zhang.
\newblock Representation theorems for {B}ackward {S}tochastic {D}ifferential
  {E}quations.
\newblock \emph{Annals of Applied Probability}, 12\penalty0 (4):\penalty0
  1390--1418, 2002.

\bibitem[Pardoux and Peng(1992)]{PP2}
Etienne Pardoux and Shige Peng.
\newblock {B}ackward {S}tochastic {D}ifferential {E}quations and quasi linear
  parabolic partial differential equations.
\newblock \emph{Lecture notes in control and information science},
  176:\penalty0 200--217, 1992.

\bibitem[Zhang(2004)]{Z}
Jianfeng Zhang.
\newblock A numerical scheme for {B}{S}{D}{E}s.
\newblock \emph{Annals of Applied Probability}, 14\penalty0 (1):\penalty0
  459--488, 2004.

\end{thebibliography}
\end{document}